\documentclass{amsart}
\usepackage{amsmath,amssymb,amsthm,mathrsfs}

\usepackage{graphicx}

\theoremstyle{plain}
\numberwithin{equation}{section}
\theoremstyle{plain}
\newtheorem{theorem}{Theorem}[section]

\newtheorem{lemma}[theorem]{Lemma}
\newtheorem{proposition}[theorem]{Proposition}

\theoremstyle{definition}

\newtheorem{example}[theorem]{Example}

\newcommand{\ands}{\quad\mbox{and}\quad}
\newcommand{\kernel}{{\mathrm{Ker}}}
\newcommand{\degr}{{\mathrm{deg}}}
\newcommand{\Dom}{{\mathrm{Dom}}}

\newcommand{\Index}{{\mathrm{Index}}}

\newcommand{\Rat}{{\mathrm{Rat}}}
\newcommand{\Ran}{{\mathrm{Ran}}}

\newcommand{\im}{{\mathrm{Im}}}
\newcommand{\re}{{\mathrm{Re}}}
\newcommand{\half}{\mbox{$\frac{1}{2}$}}
\newcommand{\tu}[1]{\textup{#1}}

\newcommand{\cP}{{\mathcal P}}
\newcommand{\cQ}{{\mathcal Q}}

\newcommand{\BC}{{\mathbb C}}
\newcommand{\BN}{{\mathbb N}}
\newcommand{\BT}{{\mathbb T}}
\newcommand{\BD}{{\mathbb D}}
\newcommand{\BP}{{\mathbb P}}
\newcommand{\BR}{{\mathbb R}}
\newcommand{\wtil}[1]{{\widetilde{#1}}}
\newcommand{\what}[1]{{\widehat{#1}}}

\newcommand{\al}{\alpha}

\newcommand{\la}{\lambda}

\newcommand{\si}{\sigma}

\newcommand{\om}{\omega}

\begin{document}

\title[Unbounded Toeplitz-like operators II: the spectrum]{A Toeplitz-like operator with rational symbol having poles on the unit circle II: the spectrum}

\author[G.J. Groenewald]{G.J. Groenewald}
\address{G.J. Groenewald, Department of Mathematics, Unit for BMI, North-West
University,
Potchefstroom, 2531 South Africa}
\email{Gilbert.Groenewald@nwu.ac.za}

\author[S. ter Horst]{S. ter Horst}
\address{S. ter Horst, Department of Mathematics, Unit for BMI, North-West
University, Potchefstroom, 2531 South Africa}
\email{Sanne.TerHorst@nwu.ac.za}

\author[J. Jaftha]{J. Jaftha}
\address{J. Jaftha, Numeracy Centre, University of Cape Town, Rondebosch 7701; Cape Town; South Africa}
\email{Jacob.Jaftha@uct.ac.za}

\author[A.C.M. Ran]{A.C.M. Ran}
\address{A.C.M. Ran, Department of Mathematics, Faculty of Science, VU Amsterdam, De Boelelaan 1081a, 1081 HV Amsterdam, The Netherlands and Unit for BMI, North-West~University, Potchefstroom, South Africa}
\email{a.c.m.ran@vu.nl}

\thanks{This work is based on the research supported in part by the National Research Foundation of South Africa (Grant Number 90670 and 93406).\\
Part of the research was done during a sabbatical of the third author, in which time several
research visits to VU Amsterdam and North-West University were made. Support from University
of Cape Town and the Department of Mathematics, VU Amsterdam is gratefully acknowledged.}

\subjclass{Primary 47B35, 47A53; Secondary 47A68}

\keywords{unbounded Toeplitz operator, spectrum, essential spectrum}

\begin{abstract}
This paper is a continuation of our study of a class of Toeplitz-like operators with a rational symbol which has a pole on the unit circle. A description of the spectrum and its various parts, i.e., point, residual and continuous spectrum, is given, as well as a description of the essential spectrum. In this case, the essential spectrum need not be connected in $\BC$. Various examples illustrate the results.
\end{abstract}

\subjclass[2010]{Primary 47B35, 47A53; Secondary 47A68}

\keywords{unbounded Toeplitz operator, spectrum, essential spectrum}

\maketitle


\section{Introduction}

This paper is a continuation of our earlier paper \cite{GtHJR1} where Toeplitz-like operators with rational symbols which may have poles on the unit circle where introduced. While the aim of \cite{GtHJR1} was to determine the Fredholm properties of such Toeplitz-like operators, in the current paper we will focus on properties of the spectrum. For this purpose we further analyse this class of Toeplitz-like operators, specifically in the case where the operators are not Fredholm.

We start by recalling the definition of our Toeplitz-like operators.
Let $\Rat$ denote the space of rational complex functions. Write $\Rat(\mathbb{T})$ and $\Rat_0(\mathbb{T})$ for the subspaces of $\Rat$ consisting of the rational functions in $\Rat$ with all poles on $\BT$ and the strictly proper rational functions in $\Rat$ with all poles on the unit circle $\BT$, respectively. For $\om \in \Rat$, possibly having poles on $\BT$, we define a Toeplitz-like operator $T_\omega (H^p \rightarrow H^p)$, for $1 < p<\infty$, as follows:
\begin{equation}\label{Toeplitz}
\Dom(T_\omega)\!=\! \left\{ g\in H^p \! \mid \! \omega g = f + \rho \mbox{ with } f\!\in\! L^p\!\!,\, \rho \!\in\!\textup{Rat}_0(\mathbb{T})\right\},\
T_\omega g = \mathbb{P}f.
\end{equation}
Here $\BP$ is the Riesz projection of $L^p$ onto $H^p$.

In \cite{GtHJR1} it was established that this operator is a densely defined, closed operator which is Fredholm if and only if $\om$ has no zeroes on $\BT$. In case the symbol $\om$ of $T_\om$ is in $\Rat(\mathbb{T})$  with no zeroes on $\BT$, i.e., $T_\om$ Fredholm, explicit formulas for the domain, kernel, range and a complement of the range were also obtained in \cite{GtHJR1}. Here we extend these results to the case that $\om$ is allowed to have zeroes on $\BT$, cf., Theorem \ref{T:Rat(T)} below. By a reduction to the case of symbols in $\Rat(\mathbb{T})$, we then obtain for general symbols in $\Rat$, in Proposition \ref{P:injectdenserange} below, necessary and sufficient conditions for $T_\om$ to be injective or have dense range, respectively.

\paragraph{\bf Main results}


Using the fact that $\la I_{H^p}-T_{\om}=T_{\la-\om}$, our extended analysis of the operator $T_{\om}$ enables us to describe the spectrum of $T_\om$, and its various parts. Our first main result is a description of the essential spectrum of $T_\om$, i.e., the set of all $\la\in\BC$ for which $\la I_{H^p}-T_{\om}$ is not Fredholm.

\begin{theorem}\label{T:main1}
Let $\om\in\Rat$. Then the essential spectrum $\si_\textup{ess}(T_{\om})$ of $T_{\om}$ is an algebraic curve in $\BC$ which is given by
\[
\si_\textup{ess}(T_{\om})=\om(\BT):=\{\om(e^{i\theta}) \mid 0\leq \theta \leq 2\pi,\, \mbox{$e^{i\theta}$ not a pole of $\om$} \}.
\]
Furthermore, the map $\la\mapsto \Index (T_{\la-\om})$ is constant on connected components of $\BC\backslash \om(\BT)$ and the intersection of the point spectrum, residual spectrum and resolvent set of $T_\om$ with $\BC\backslash \om(\BT)$ coincides with sets of $\la\in\BC\backslash \om(\BT)$ with $\Index (T_{\la-\om})$ being strictly positive, strictly negative and zero, respectively.
\end{theorem}

Various examples, specifically in Section \ref{S:ExEssSpec}, show that the algebraic curve $\om(\BT)$, and thus the essential spectrum of $T_\om$, need not be connected in $\BC$.

Our second main result provides a description of the spectrum of $T_{\om}$ and its various parts. Here and throughout the paper $\cP$ stands for the subspace of $H^p$ consisting of all polynomials and $\cP_k$ for the subspace of $\cP$ consisting of all polynomials of degree at most $k$.

\begin{theorem}\label{T:main2}
Let $\om\in\Rat$, say $\om=s/q$ with $s,q\in\cP$ co-prime. Define
\begin{equation}\label{Kq0-}
\begin{aligned}
k_q&=\sharp\{\mbox{roots of $q$ inside $\overline{\BD}$}\}=\sharp\{\mbox{poles of $\la-\om$ inside $\overline{\BD}$}\},\\
k_\la^-&=\sharp\{\mbox{roots of $\la q-s$ inside $\BD$}\}=\sharp\{\mbox{zeroes of $\la-\om$ inside $\BD$}\},\\
k_\la^0&=\sharp\{\mbox{roots of $\la q-s$ on $\BT$}\}=\sharp\{\mbox{zeroes of $\la-\om$ on $\BT$}\},
\end{aligned}
\end{equation}
where in all these sets multiplicities of the roots, poles and zeroes are to be taken into account. Then the resolvent set $\rho(T_\om)$, point spectrum $\si_\textup{p}(T_\om)$, residual spectrum $\si_\textup{r}(T_\om)$ and continuous spectrum $\si_\textup{c}(T_\om)$ of $T_\om$ are given by
\begin{equation}\label{specparts}
\begin{aligned}
\rho(T_{\om})&=\{\la\in\BC \mid k_\la^0=0 \mbox{ and } k_q=k_\la^-\},\\
\si_\textup{p}(T_{\om})=\{\la\in\BC &\mid  k_q>k_\la^-+k_\la^0\},\quad
\si_\textup{r}(T_{\om})=\{\la\in\BC \mid k_q<k_\la^-\},\\
\si_\textup{c}(T_{\om})&=\{\la\in\BC \mid k_\la^0>0 \mbox{ and } k_\la^- \leq k_q\leq k_\la^- + k_\la^0\}.
\end{aligned}
\end{equation}
Furthermore, $\si_\textup{ess}(T_\om)=\om(\BT)=\{\la\in\BC \mid k_\la^0>0\}$.
\end{theorem}

Again, in subsequent sections various examples are given that illustrate these results. In particular, examples are given where $T_{\om}$ has a bounded resolvent set, even with an empty resolvent set. This is in sharp contrast to the case where $\om$ has no poles on the unit circle $\BT$. For in this case the operator is bounded, the resolvent set is a nonempty unbounded set and the spectrum a compact set, and the essential spectrum is connected.

Both Theorems \ref{T:main1} and \ref{T:main2} are proven in Section \ref{S:Spectrum}.

\paragraph{\bf Discussion of the literature}

In the case of a bounded selfadjoint Toeplitz operator on $\ell^2$, Hartman and Wintner in \cite{HW50} showed that the point spectrum is empty when the symbol is real and rational and posed the problem of specifying the spectral properties of such a Toeplitz operator. Gohberg in \cite{G52}, and more explicitly in \cite{G67}, showed that a bounded Toeplitz operator with continuous symbol is Fredholm exactly when the symbol has no zeroes on $\BT$, and in this case the index of the operator coincides with the negative of the winding number of the symbol with respect to zero. This implies immediately that the essential spectrum of a Toeplitz operator with continuous symbol is the image of the unit circle.

Hartman and Wintner in \cite{HW54} followed up their earlier question by showing that in the case where the symbol, $\varphi$, is a bounded real valued function on $\BT$, the spectrum of the Toeplitz operator on $H^2$ is contained in the interval bounded by the essential lower and upper bounds of $\varphi$ on $\BT$ as well as that the point spectrum is empty whenever $\varphi$ is not a constant. Halmos, after posing in \cite{H63} the question whether the spectrum of a Toeplitz operator is connected, with Brown in \cite{BH64} showed that the spectrum cannot consist of only two points. Widom, in \cite{W64}, established that bounded Toeplitz operators on $H^2$ have connected spectrum, and later extended the result for general $H^p$, with $1 \leq p \leq \infty$. That the essential (Fredholm) spectrum of a bounded Toeplitz operator in $H^2$ is connected was shown by Douglas in \cite{D98}. For the case of bounded Toeplitz operators in $H^p$ it is
posed as an open question in B\"ottcher and Silbermann in \cite[Page 70]{BS06} whether the essential (Fredholm) spectrum of a Toeplitz operator in $H^p$ is necessarily connected.
 Clark, in \cite{C67}, established conditions on the argument of the symbol $\varphi$ in the case $\varphi\in L^q, q \geq 2$ that would give the kernel index of the Toeplitz operator with symbol
 $\varphi$ on $L^p$, where $\frac{1}{p} + \frac{1}{q} = 1$, to be $m\in\BN$.

Janas, in \cite{J91},
discussed unbounded Toeplitz operators on the Bargmann-Siegel space and
showed that $\sigma_\tu{ess}(T_\varphi) \subset \cap_{R>0} \textrm{ closure } \{\varphi (z): \vert z\vert\geq R\}$.

\paragraph{\bf Overview} The paper is organized as follows. Besides the current introduction, the paper consists of five sections. In Section \ref{S:Review} we extend a few results concerning the operator $T_\om$ from \cite{GtHJR1} to the case where $T_\om$ need not be Fredholm. These results are used in Section \ref{S:Spectrum} to compute the spectrum of $T_{\om}$ and various of its subparts, and by doing so we prove the main results, Theorems \ref{T:main1} and \ref{T:main2}. The remaining three sections contain examples that illustrate our main results and show in addition that the resolvent set can be bounded, even empty, and that the essential spectrum can be disconnected in $\BC$.

\paragraph{\bf Figures}
We conclude this introduction with a remark on the figures in this paper illustrating the spectrum and essential spectrum for several examples. The color coding in these figures is as follows: the white region is the resolvent set, the black curve is the essential spectrum, and the colors in the other regions codify the Fredholm index, where red indicates index $2$, blue indicates index $1$, cyan indicates index $-1$, magenta indicates index $-2$.

\section{Review and new results concerning $T_\omega$}\label{S:Review}

In this section we recall some results concerning the operator $T_\om$ defined in \eqref{Toeplitz} that were obtained in \cite{GtHJR1} and will be used in the present paper to determine spectral properties of $T_\om$. A few new features are added as well, specifically relating to the case where $T_\om$ is not Fredholm.

The first result provides necessary and sufficient conditions for $T_{\om}$ to be Fredholm, and gives a formula for the index of $T_\om$ in case $T_\om$ is Fredholm.

\begin{theorem}[Theorems 1.1 and 5.4 in \cite{GtHJR1}]\label{T:recall1}
Let $\om\in \Rat$. Then $T_\om$ is Fredholm if and only if $\om$ has no zeroes on $\BT$. In case $T_\om$ is Fredholm, the Fredholm index of $T_\om$ is given by
\[
\Index (T_\om) = \sharp \left\{\begin{array}{l}\!\!\!
 \textrm{poles of } \om \textrm{ in }\overline{\BD} \textrm{ multi.}\!\!\! \\
\!\!\!\textrm{taken into account}\!\!\!
\end{array}\right\}  -
\sharp \left\{\begin{array}{l}\!\!\! \textrm{zeroes of } \om\textrm{ in }\BD  \textrm{ multi.}\!\!\! \\
\!\!\!\textrm{taken into account}\!\!\!
\end{array}\right\},
\]
and  $T_\om$ is either injective or surjective. In particular, $T_\om$ is injective, invertible or surjective if and only if $\Index(T_\om)\leq 0$, $\Index(T_\om)=0$ or $\Index(T_\om)\geq 0$, respectively.
\end{theorem}

Special attention is given in \cite{GtHJR1} to the case where $\om$ is in $\Rat(\BT)$, since in that case the kernel, domain and range can be computed explicitly; for the domain and range this was done under the assumption that $T_\om$ is Fredholm. In the following result we collect various statements from Proposition 4.5 and Theorems 1.2 and 4.7 in \cite{GtHJR1} and extend to or improve some of the claims regarding the case that $T_\om$ is not Fredholm.

\begin{theorem}\label{T:Rat(T)}
Let $\om\in \Rat(\BT)$, say $\om=s/q$ with $s,q\in\cP$ co-prime. Factor $s=s_-s_0s_+$ with $s_-$, $s_0$ and $s_+$ having roots only inside, on, or outside $\BT$. Then
\begin{equation}\label{DomRanId}
\begin{aligned}
&\qquad \kernel (T_\omega) = \left\{r_0/s_+ \mid \deg(r_0) < \deg(q) - \deg(s_-s_0) \right\};\\
&\Dom(T_\om)=qH^p+\cP_{\deg(q)-1}; \quad
\Ran(T_\om)=s H^p+\wtil\cP,
\end{aligned}
\end{equation}
where $\wtil\cP$ is the subspace of $\cP$ given by
\begin{equation}\label{tilP}
\wtil\cP = \{ r\in\cP \mid r q = r_1 s + r_2 \mbox{ for } r_1,r_2\in\mathcal{P}_{\deg(q)-1}\}\subset \cP_{\deg(s)-1}.
\end{equation}
Furthermore, $H^p=\overline{\Ran(T_\om)} + \wtil{\cQ}$ forms a direct sum decomposition of $H^p$, where
\begin{equation}\label{tilQ}
\wtil\cQ=\cP_{k-1}\quad \mbox{with}\quad k=\max\{\deg(s_-)-\deg(q) , 0\},
\end{equation}
following the convention $\cP_{-1}:=\{0\}$.
\end{theorem}

The following result will be useful in the proof of Theorem \ref{T:Rat(T)}.

\begin{lemma}\label{L:closure}
Factor $s\in\cP$ as $s=s_-s_0s_+$ with $s_-$, $s_0$ and $s_+$ having roots only inside, on, or outside $\BT$. Then $sH^p =s_-s_0 H^p$ and $\overline{s H^p}= s_- H^p$.
\end{lemma}

\begin{proof}[\bf Proof]
Since $s_+$ has no roots inside $\overline{\BD}$, we have $s_+ H^p=H^p$. Furthermore, $s_0$ is an $H^\infty$ outer function (see, e.g., \cite{N}, Example 4.2.5)
so that $\overline{s_0 H^p}=H^p$.
Since $s_-$ has all it's roots inside $\BD$, $T_{s_-}:H^p\to H^p$ is an injective operator with closed range. Consequently, we have
\[
\overline{s H^p}
=\overline{s_- s_0 s_+ H^p}
=\overline{s_- s_0 H^p}
=s_- \overline{ s_0 H^p}=s_- H^p,
\]
as claimed.
\end{proof}

\begin{proof}[\bf Proof of Theorem \ref{T:Rat(T)}]
In case $T_\om$ is Fredholm, i.e., $s_0$ constant, all statements follow from Theorem 1.2 in \cite{GtHJR1}. Without the Fredholm condition, the formula for $\kernel(T_\om)$ follows from \cite[Lemma 4.1]{GtHJR1} and for $\Dom(T_\om)$ and $\Ran(T_\om)$ Proposition 4.5 of \cite{GtHJR1} provides
\begin{equation}\label{DomRanIncl}
\begin{aligned}
qH^p+\cP_{\deg(q)-1}&\subset \Dom(T_\om);\\
T_\om (qH^p+\cP_{\deg(q)-1})&=s H^p+\wtil\cP\subset \Ran(T_\om).
\end{aligned}
\end{equation}
Thus in order to prove \eqref{DomRanId}, it remains to  show that $\Dom(T_{\om})\subset q H^p +\cP_{\deg(q)-1}$.

Assume $g\in \Dom(T_{\om})$. Thus there exist $h\in H^p$ and $r\in \cP_{\deg(q)-1}$ so that $s g= q h +r$. Since $s$ and $q$ are co-prime, there exist $a,b\in\cP$ such that $s a+ q b\equiv 1$. Next write $ar=q r_1+r_2$ for $r_1,r_2\in\cP$ with $\deg(r_2)<\deg (q)$. Thus $sg=q h +r=q h+ q br + s ar=q(h+br+sr_1)+ s r_2$. Hence $g=q(h+br+sr_1)/s +r_2$. We are done if we can show that $\wtil{h}:=(h+br+sr_1)/s$ is in $H^p$.

The case where $g$ is rational is significantly easier, but still gives an idea of the complications that arise, so we include a proof. Hence assume $g\in\Rat\cap H^p$. Then $h=(s g-r)/q$ is also in $\Rat \cap H^p$, and $\wtil{h}$ is also rational. It follows that $q(h+br+sr_1)/s=q \wtil{h} =g-r_2\in \Rat \cap H^p$ and thus cannot have poles in $\overline{\BD}$. Since $q$ and $s$ are co-prime and $h$ cannot have poles inside $\overline{\BD}$, it follows that $\wtil{h}=(h+br+sr_1)/s$ cannot have poles in $\overline{\BD}$. Thus $\wtil{h}$ is a rational function with no poles in $\overline{\BD}$, which implies $\wtil{h}\in H^p$.

Now we prove the claim for the general case. Assume $q\wtil{h} +r_2=g\in H^p$, but $\wtil{h}=(h+br+sr_1)/s\not \in H^p$, i.e., $\wtil{h}$ is not analytic on $\BD$ or $\int_{\BT} |\wtil{h}(z)|^p\textup{d}z=\infty$. Set $\widehat{h}=h+br+sr_1\in H^p$, so that $\wtil{h}=\what{h}/s$. We first show $\wtil{h}$ must be analytic on $\BD$. Since $\wtil{h}=\what{h}/s$ and $\what{h}\in H^p$, $\wtil{h}$ is analytic on $\BD$ except possibly at the roots of $s$. However, if $\wtil{h}$ would not be analytic at a root $z_0\in\BD$ of $s$, then also $g= q \wtil{h}+r_2$ should not be analytic at $z_0$, since $q$ is bounded away from 0 on a neighborhood of $z_0$, using that $s$ and $q$ are co-prime. Thus $\wtil{h}$ is analytic on $\BD$. It follows that $\int_{\BT} |\wtil{h}(z)|^p\textup{d}z=\infty$.

Since $s$ and $q$ are co-prime, we can divide $\BT$ as $\BT_1\cup \BT_2$ with $\BT_1\cap \BT_2=\emptyset$ and each of $\BT_1$ and $\BT_2$ being nonempty unions of intervals, with $\BT_1$ containing all roots of $s$ on $\BT$ as interior points and $\BT_2$ containing all roots of $q$ on $\BT$ as interior points. Then there exist $N_1,N_2>0$ such that $|q(z)|> N_1$ on $\BT_1$ and $|s(z)|>N_2$ on $\BT_2$.
Note that
\begin{align*}
\int_{\BT_2}|\wtil{h}(z)|^p \textup{d}z
&=\int_{\BT_2}|\what{h}(z)/s(z)|^p \textup{d}z
\leq  N_2^{-p} \int_{\BT_2}|\what{h}(z)|^p \textup{d}z
\leq N_2^{-p} \|\what{h}\|^p_{H^p}<\infty.
\end{align*}
Since $\int_{\BT}|\wtil{h}(z)|^p \textup{d}z=\infty$ and $\int_{\BT_2}|\wtil{h}(z)|^p \textup{d}z<\infty$, it follows that $\int_{\BT_1}|\wtil{h}(z)|^p \textup{d}z=\infty$. However, since $|q(z)|> N_1$ on $\BT_1$, this implies that
\begin{align*}
\|g-r_2\|^p_{H^p}&=\int_{\BT}|g(z)-r_2(z)|^p \textup{d}z
=\int_{\BT}|q(z)\wtil{h}(z)|^p \textup{d}z
\geq  \int_{\BT_1}|q(z) \wtil{h}(z)|^p \textup{d}z\\
&\geq N_1^p \int_{\BT_1}|\wtil{h}(z)|^p \textup{d}z =\infty,
\end{align*}
in contradiction with the assumption that $g\in H^p$. Thus we can conclude that $\wtil{h}\in H^p$ so that $g=q \wtil{h}+r_2$ is in $q H^p+\cP_{\deg(q)-1}$.

It remains to show that $H^p=\overline{\Ran(T_\om)}+\wtil{\cQ}$ is a direct sum decomposition of $H^p$. Again, for the case that $T_\om$ is Fredholm this follows from \cite[Theorem 1.2]{GtHJR1}. By the preceding part of the proof we know, even in the non-Fredholm case, that $\Ran(T_\om)=s H^p+\wtil{\cP}$. Since $\wtil{\cP}$ is finite dimensional, and thus closed, we have
\[
\overline{\Ran(T_\om)}=\overline{s H^p}+\wtil{\cP}= s_- H^p + \wtil{\cP},
\]
using Lemma \ref{L:closure} in the last identity. We claim that
\[
\overline{\Ran(T_\om)}=s_- H^p + \wtil{\cP}=s_- H^p + \wtil{\cP}_-,
\]
where $\wtil{\cP}_-$ is defined by
\[
\wtil{\cP}_{-}:=\{r\in\cP \mid qr=r_1 s_- +r_2\mbox{ for }r_1,r_2\in\cP_{\deg(q)-1}\}\subset \cP_{\deg(s_-)-1}.
\]
Once the above identity for $\overline{\Ran(T_\om)}$ is established, the fact that $\wtil{\cQ}$ is a complement of $\overline{\Ran(T_\om)}$ follows directly by applying Lemma 4.8 of \cite{GtHJR1} to $s=s_-$.

We first show that $\overline{\Ran(T_\om)}=s_- H^p +\wtil{\cP}$ is contained in $s_- H^p+\wtil{\cP}_-$. Let $g=s_- h+r$ with $h\in H^p$ and $r\in\wtil{\cP}$, say $qr=r_1s +r_2$ with $r_1,r_2\in\cP_{\deg(q)-1}$. Write $r_1 s_0s_+=\wtil{r}_1 q + \wtil{r}_2$ with $\deg(\wtil{r}_2)<\deg(q)$. Then
\[
qr= r_1 s_- s_0s_+ +r_2=q\wtil{r}_1 s_-+\wtil{r}_2s_- +r_2,\mbox{ so that }
q(r-\wtil{r}_1s_-)= \wtil{r}_2s_- + r_2,
\]
with $r_2,\wtil{r_2}\in\cP_{\deg(q)-1}$. Thus $r-\wtil{r}_1s_-\in\wtil{\cP}_-$. Therefore, we have
\[
g=s_-(h+\wtil{r}_1)+(r-\wtil{r}_1s_-)\in s_- H^p +\wtil{\cP}_-,
\]
proving that $\overline{\Ran(T_\om)}\subset s_- H^p+\wtil{\cP}_-$.

For the reverse inclusion, assume $g=s_-h+r\in s_- H^p+\wtil{\cP}_-$. Say $qr=r_1 s_- + r_2$ with $r_1,r_2\in\cP_{\deg(q)-1}$. Since $s_0s_+$ and $q$ are co-prime and $\deg(r_1)<\deg(q)$ there exit polynomials $\wtil{r}_1$ and $\wtil{r}_2$ with $\deg(\wtil{r}_1)<\deg(q)$ and $\deg(\wtil{r}_2)<\deg(s_0s_+)$ that satisfy the B\'ezout equation $\wtil{r}_1 s_0 s_+ + \wtil{r}_2 q=r_1$. Then
\[
\wtil{r}_1 s+ r_2 = \wtil{r}_1 s_0s_+s_- + r_2= (r_1-\wtil{r}_2 q)s_- + r_2= r_1 s_- + r_2  -q \wtil{r}_2 s_- = q(r-\wtil{r}_2 s_-).
\]
Hence $r-\wtil{r}_2 s_-$ is in $\wtil{\cP}$, so that $g= s_- h+ r= s_-(h+\wtil{r}_2)+ (r-\wtil{r}_2 s_-)\in s_- H^p +\wtil{\cP}$. This proves the reverse inclusion, and hence completes the proof of Theorem \ref{T:Rat(T)}.
 \end{proof}

The following result makes precise when $T_\om$ is injective and when $T_\om$ has dense range, even in the case where $T_\om$ is not Fredholm.

\begin{proposition}\label{P:injectdenserange}
Let $\om\in \Rat$. Then $T_\om$ is injective if and only if
\[
\sharp \left\{\begin{array}{l}\!\!\!
 \textrm{poles of } \om \textrm{ in }\overline{\BD} \textrm{ multi.}\!\!\! \\
\!\!\!\textrm{taken into account}\!\!\!
\end{array}\right\}  \leq
\sharp \left\{\begin{array}{l}\!\!\! \textrm{zeroes of } \om\textrm{ in }\overline{\BD}  \textrm{ multi.}\!\!\! \\
\!\!\!\textrm{taken into account}\!\!\!
\end{array}\right\}.
\]
Moreover, $T_\om$ has dense range if and only if
\[
\sharp \left\{\begin{array}{l}\!\!\!
 \textrm{poles of } \om \textrm{ in }\overline{\BD} \textrm{ multi.}\!\!\! \\
\!\!\!\textrm{taken into account}\!\!\!
\end{array}\right\}  \geq
\sharp \left\{\begin{array}{l}\!\!\! \textrm{zeroes of } \om\textrm{ in }\BD  \textrm{ multi.}\!\!\! \\
\!\!\!\textrm{taken into account}\!\!\!
\end{array}\right\}.
\]
In particular, $T_\om$ is injective or has dense range.
\end{proposition}

\begin{proof}[\bf Proof]
First assume $\om\in \Rat(\BT)$. By Corollary 4.2 in \cite{GtHJR1}, $T_\om$ is injective if and only if the number of zeroes of $\om$ inside $\overline{\BD}$ is greater than or equal to the number of poles of $\om$, in both cases with multiplicity taken into account. By Theorem \ref{T:Rat(T)}, $T_\om$ has dense range precisely when $\wtil{\cQ}$ in \eqref{tilQ} is trivial. The latter happens if and only if the number of poles of $\om$ is greater than or equal to the number of zeroes of $\om$ inside $\BD$, again taking multiplicities into account. Since in this case all poles of $\om$ are in $\BT$, our claim follows for $\om\in \Rat(\BT)$.

Now we turn to the general case, i.e., we assume $\om\in\Rat$. In the remainder of the proof, whenever we speak of numbers of zeroes or poles, this always means that the respective multiplicities are to be taken into account. Recall from \cite[Lemma 5.1]{GtHJR1} that we can factor $\om(z)= \om_-(z)z^\kappa \om_0(z) \om_+(z)$ with $\om_-,\om_0,\om_+\in\Rat$, $\om_-$ having no poles or zeroes outside $\BD$, $\om_+$ having no poles or zeroes inside $\overline{\BD}$ and $\om_0$ having poles and zeroes only on $\BT$, and $\kappa$ the difference between the number of zeroes of $\om$ in $\BD$ and the number of poles of $\om$ in $\BD$. Moreover, we have $T_\om=T_{\om_-}T_{z^\kappa \om_0} T_{\om_+}$ and $T_{\om_-}$ and $T_{\om_+}$ are boundedly invertible on $H^p$. Thus $T_\om$ is injective or has closed range if and only it $T_{z^\kappa\om_0}$ is injective or has closed range, respectively.

Assume $\kappa \geq 0$. Then $z^\kappa \om_0\in \Rat(\BT)$ and the results for the case that the symbol is in $\Rat(\BT)$ apply. Since the zeroes and poles of $\om_0$ coincide with the zeroes and poles of $\om$ on $\BT$, it follows that the number of poles of $z^\kappa \om_0$ is equal to the number of poles of $\om$ on $\BT$ while the number of zeroes of $z^\kappa \om_0$ is equal to $\kappa$ plus the number of zeroes of $\om$ on $\BT$ which is equal to the number of zeroes of $\om$ in $\overline{\BD}$ minus the number of poles of $\om$ in $\BD$. It thus follows that $T_{z^\kappa \om_0}$ is injective, and equivalently $T_\om$ is injective, if and only if the number of zeroes of $\om$ in $\overline{\BD}$ is greater than or equal to the number of poles of $\om$ in $\overline{\BD}$, as claimed.

Next, we consider the case where $\kappa<0$. In that case $T_{z^\kappa \om_0}=T_{z^\kappa}T_{\om_0}$, by Lemma 5.3 of \cite{GtHJR1}. We prove the statements regarding injectivity and $T_\om$ having closed range separately.

First we prove the injectivity claim for the case where $\kappa<0$. Write $\om_0 =s_0/q_0$ with $s_0,q_0\in\cP$ co-prime. Note that all the roots of $s_0$ and $q_0$ are on $\BT$. We need to show that $T_{z^\kappa \om_0}$ is injective if and only if $\deg(s_0) \geq \deg(q_0)-\kappa$ (recall, $\kappa$ is negative).

Assume $\deg(s_0)+\kappa \geq \deg(q_0)$. Then $\deg(s_0) > \deg(q_0)$, since $\kappa<0$, and thus $T_{\om_0}$ is injective. We have $\kernel(T_{z^{\kappa}})=\cP_{|\kappa|-1}$. So it remains to show $\cP_{|\kappa|-1} \cap \Ran (T_{\om_0})=\{0\}$. Assume $r\in\cP_{|\kappa|-1}$ is also in $\Ran (T_{\om_0})$. So, by Lemma 2.3 in \cite{GtHJR1}, there exist $g\in H^p$ and $r'\in\cP_{\deg(q_0)-1}$ so that $s_0 g=q_0 r+ r'$, i.e., $g=(q_0 r+ r')/s_0$. This shows that $g$ is in $\Rat(\BT)\cap H^p$, which can only happen in case $g$ is a polynomial. Thus, in the fraction $(q_0 r+ r')/s_0$, all roots of $s_0$ must cancel against roots of $q_0 r+ r'$. However, since $\deg(s_0)+\kappa \geq \deg(q_0)$, with $\kappa<0$, $\deg(r)<\deg |\kappa|-1$ and $\deg(r')<\deg(q_0)$, we have $\deg(q_0 r + r')<\deg(s_0)$ and it is impossible that all roots of $s_0$ cancel against roots of $q_0 r + r'$, leading to a contradiction. This shows $\cP_{|\kappa|-1} \cap \Ran (T_{\om_0})=\{0\}$, which implies $T_{z^{\kappa\om_0}}$ is injective. Hence also $T_\om$ is injective.

Conversely, assume $\deg(s_0)+\kappa < \deg(q_0)$, i.e., $\deg(s_0)< \deg(q_0)+|\kappa|=:b$, since $\kappa<0$. Then
\[
s_0\in \cP_{b-1}=q_0 \cP_{|\kappa|-1} +\cP_{\deg(q_0)-1}.
\]
This shows there exist $r\in\cP_{|\kappa|-1}$ and $r'\in\cP_{\deg(q_0)-1}$ so that $s_0= q_0 r+ r'$. In other words, the constant function $g\equiv 1\in H^p$ is in $\Dom (T_{\om_0})$ and $T_{\om_0}g=r\in \cP_{|\kappa|-1}=\kernel (T_{z^\kappa})$, so that $g\in \kernel (T_{z^\kappa \om_0})$. This implies $T_\om$ is not injective.

Finally, we turn to the proof of the dense range claim for the case $\kappa<0$. Since $\kappa<0$ by assumption, $\om$ has more poles in $\overline{\BD}$ (and even in $\BD$) than zeroes in $\BD$. Thus to prove the dense range claim in this case, it suffices to show that $\kappa<0$ implies that $T_{z^\kappa\om_0}$ has dense range. We have $T_{z^\kappa \om_0}=T_{z^\kappa}T_{\om_0}$ and $T_{z^\kappa}$ is surjective. Also, $\om_0\in \Rat(\BT)$ has no zeroes inside $\BD$. So the proposition applies to $\om_0$, as shown in the first paragraph of the proof, and it follows that $T_{\om_0}$ has dense range. But then also $T_{z^\kappa \om_0}=T_{z^\kappa}T_{\om_0}$ has dense range, and our claim follows.
\end{proof}

\section{The spectrum of $T_\omega$}\label{S:Spectrum}

In this section we determine the spectrum and various subparts of the spectrum of $T_\om$ for the general case, $\om\in\Rat$, as well as some refinements for the case where $\om\in\Rat(\BT)$ is proper. In particular, we prove our main results, Theorems \ref{T:main1} and \ref{T:main2}.

Note that for $\om\in\Rat$ and $\la\in\BC$ we have $\la I-T_\om=T_{\la-\om}$. Thus we can relate questions on the spectrum of $T_\om$ to question on injectivity, surjectivity, closed rangeness, etc.\ for Toeplitz-like operators with an additional complex parameter. By this observation, the spectrum of $T_\om$, and its various subparts, can be determined using the results of Section \ref{S:Review}.

\begin{proof}[\bf Proof of Theorem \ref{T:main1}]
Since $\la I-T_\om=T_{\la-\om}$ and $T_{\la-\om}$ is Fredholm if and only if $\la-\om$ has no zeroes on $\BT$, by Theorem \ref{T:recall1}, it follows that $\la$ is in the essential spectrum if and only if $\la=\om(e^{i\theta})$ for some $0\leq \theta\leq 2\pi$. This shows that $\si_\textup{ess}(T_\om)$ is equal to $\om(\BT)$.

To see that $\om(\BT)$ is an algebraic curve, let $\omega=s/q$ with  $s,q\in\cP$ co-prime. Then $\la=u+iv=\om(z)$ for $z=x+iy$ with $x^2+y^2=1$ if and only if $\la q(z)-s(z)=0$. Denote
$q(z)=q_1(x,y)+iq_2(x,y)$ and $s(z)=s_1(x,y)+is_2(x,y)$, where $z=x+iy$ and the functions
$q_1, q_2, s_1, s_2$ are real polynomials in two variables. Then $\la=u+iv$ is on the curve $\om(\BT)$ if and only if
\begin{align*}
q_1(x,y)u-q_2(x,y)v&=s_1(x,y),\\
q_2(x,y)u+q_1(x,y)v&=s_2(x,y),\\
x^2+y^2&=1.
\end{align*}
Solving for $u$ and $v$, this is equivalent to
\begin{align*}
(q_1(x,y)^2+q_2(x,y)^2)u-(q_1(x,y)s_1(x,y)+q_2(x,y)s_2(x,y))&=0,\\
(q_1(x,y)^2+q_2(x,y)^2)v-(q_1(x,y)s_2(x,y)-q_2(x,y)s_1(x,y))&=0,\\
x^2+y^2&=1.
\end{align*}
This describes an algebraic curve in the plane.

For $\lambda$ in the complement of the curve $\om(\BT)$ the operator $\lambda I -T_\om=T_{\la-\om}$ is
Fredholm, and according to Theorem \ref{T:recall1} the index is given by
$$
\Index (\lambda-T_\om)=
\sharp\{\textrm{ poles of } \om \textrm{ in } \overline{\BD}\}-
 \sharp\{\textrm{zeroes of } \om-\lambda \textrm{ inside }\BD\},
$$
taking the multiplicities of the poles and zeroes into account. Indeed, $\lambda - \om = \frac{\lambda q - s}{q}$ and since $q$ and $s$ are co-prime, $\lambda q - s$ and $q$ are also co-prime. Thus Theorem \ref{T:recall1} indeed applies to $T_{\la-\om}$. Furthermore, $\lambda - \om$ has the same poles as $\om$, i.e., the roots of $q$. Likewise, the zeroes of $\la-\om$ coincide with the roots of the polynomial $\lambda q - s$. Since the roots of this polynomial depend continuously on the parameter $\lambda$ the number of them is constant on connected components of the complement of the curve $\omega(\BT)$.

That the index is constant on connected components of the complement of the essential spectrum in fact holds for any unbounded densely defined operator (see \cite[Theorem VII.5.2]{S71}; see also \cite[Proposition XI.4.9]{C90} for the bounded case; for a much more refined analysis of this point see \cite{FK}).

Finally, the relation between the index of $T_{\la-\om}$ and $\la$ being in the resolvent set, point spectrum or residual spectrum follows directly by applying the last part of Theorem \ref{T:recall1} to $T_{\la-\om}$.
\end{proof}

Next we prove Theorem \ref{T:main2} using some of the new results on $T_\om$ derived in Section \ref{S:Review}.

\begin{proof}[\bf Proof of Theorem \ref{T:main2}]
That the two formulas for the numbers $k_q$, $k_\la^-$ and $k_\la^0$ coincides follows from the analysis in the proof of Theorem \ref{T:main1}, using the co-primeness of $\la q -s$ and $q$. By Theorem \ref{T:recall1}, $T_{\la-\om}$ is Fredholm if and only if $k_\la^0=0$, proving the formula for $\si_\textup{ess}(T_\om)$. The formula for the resolvent set follows directly from the fact that the resolvent set is contained in the complement of $\si_\textup{ess}(T_\om)$, i.e., $k_\la^0=0$, and that it there coincides with the set of $\la$'s for which the index of $T_{\la-\om}$ is zero, together with the formula for $\Index(T_{\la-\om})$ obtained in Theorem \ref{T:recall1}.

The formulas for the point spectrum and residual spectrum follow by applying the criteria for injectivity and closed rangeness of Proposition \ref{P:injectdenserange} to $T_{\la-\om}$ together with the fact that $T_{\la-\om}$ must be either injective or have dense range.

For the formula for the continuous spectrum, note that $\si_\textup{c}(T_\om)$ must be contained in the essential spectrum, i.e., $k_\la^0>0$. The condition $k_\la^- \leq k_q\leq k_\la^- + k_\la^0$ excludes precisely that $\la$ is in the point or residual spectrum.
\end{proof}

For the case where $\om\in\Rat(\BT)$ is proper we can be a bit more precise.

\begin{theorem}\label{T:spectrum2}
Let $\om \in\textup{Rat}(\mathbb{T})$ be proper, say $\om=s/q$ with $s,q\in\cP$ co-prime. Thus $\degr( s) \leq \degr( q)$ and all roots of $q$ are on $\BT$. Let $a$ be the leading coefficient of $q$ and $b$ the coefficient of $s$ corresponding to the monomial $z^{\deg(q)}$, hence $b=0$ if and only if $\om$ is strictly proper. Then $\si_\textup{r}(T_\om)=\emptyset$, and the point spectrum is given by
\[
\si_\textup{p}(T_\om)=\om(\BC\backslash \overline{\BD}) \cup \{b/a\}.
\]
Here $\om(\BC\backslash \overline{\BD})=\{\om (z) \mid z\in \BC\backslash \overline{\BD}\}$.
In particular, if $\om$ is strictly proper, then $0=b/a$ is in $\si_\textup{p}(T_\om)$.
Finally,
\[
\si_\textup{c}(T_\om)=\{\lambda\in\mathbb{C} \mid k_\la^0 >0 \mbox{ and all roots of } \lambda q-s \mbox{ are in }\overline{\BD} \}.
\]
\end{theorem}

\begin{proof}[\bf Proof]
Let $\om = s/q\in\Rat(\BT)$ be proper with $s,q\in\cP$ co-prime. Then $k_q=\deg(q)$. Since $\degr(s) \leq \deg(q)$, for any $\la\in\BC$ we have
\[
k_\la^-+k_\la^0\leq \deg(\la q-s)\leq \deg(q)=k_q.
\]
It now follows directly from \eqref{specparts} that $\si_\textup{r}(T_\om)=\emptyset$ and $\si_\textup{c}(T_\om)=\{\lambda \in\mathbb{C}\mid k_\lambda^0 >0, k_\lambda^-+k_\lambda^0=\deg(q)\}$. To determine the point spectrum, again using \eqref{specparts}, one has to determine when strict inequality occurs. We have $\deg(\la q-s)<\deg(q)$ precisely when the leading coefficient of $\la q$ is cancelled in $\la q-s$ or if $\la=0$ and $\deg(s)<\deg(q)$. Both cases correspond to $\la=b/a$. For the other possibility of having strict inequality, $k_\la^-+k_\la^0<\deg(\la q-s)$, note that this happens precisely when $\la q-s$ has a root outside $\overline{\BD}$, or equivalently $\la=\om(z)$ for a $z\not\in \overline{\BD}$.
\end{proof}

\section{The spectrum may be unbounded, the resolvent set empty}
\label{S:Examples1}

In this section we present some first examples, showing that the spectrum can be unbounded
and the resolvent set may be empty.

\begin{example}\label{E:spectrum2}
Let $\om(z) = \frac{z - \alpha}{z - 1}$ for some $1\neq \alpha\in\BC$,  say $\alpha=a+ib$, with $a$ and $b$ real. Let $L\subset\BC$ be the line given by
\begin{equation}\label{Line}
L=\{z=x+iy\in\BC \mid 2by = (a^2 + b^2 - 1) + (2 - 2a)x \}
\end{equation}
Then we have
\begin{align*}
  \rho(T_\om)=\om(\BD),\quad & \sigma_\textup{ess} (T_\om)=\om(\BT)=L =\si_\tu{c}(T_\om), \\
  \si_\tu{p}(T_\om)&=\om(\BC\backslash\overline{\BD}),\quad  \si_\tu{r}(T_{\om})=\emptyset.
\end{align*}
Moreover, the point spectrum of $T_\om$ is the open half plane determined by $L$ that contains $1$ and the resolvent set of $T_\om$ is the other open half plane determined by $L$.\medskip

\begin{figure}
\begin{center}
\includegraphics[height=4cm]{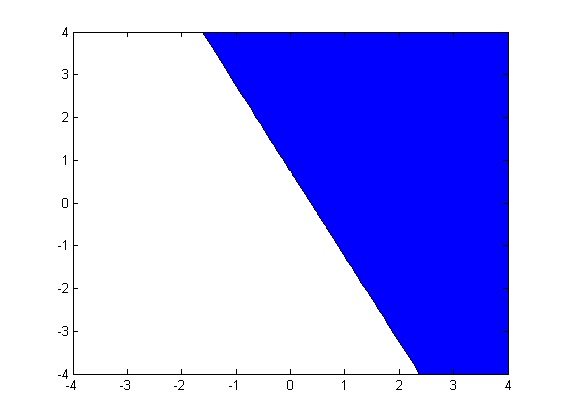}
\\
\caption
{Spectrum of $T_\om$ where $\om(z)=\frac{z-\alpha}{z-1}$, with $\alpha=-\frac{i}{2}$.}
\end{center}
\end{figure}

To see that these claims are true note that for $\lambda\not= 1$
\[
\lambda - \om(z) = \frac{z(\lambda - 1) + \alpha - \lambda}{z - 1} = \frac{1}{\lambda - 1}\frac{z + \frac{\alpha - \lambda}{\lambda - 1}}{z -1},
\]
while for $\lambda = 1$ we have $\lambda - \om(z) = \frac{\alpha - \lambda}{z - 1}$. Thus $\lambda = 1\in\sigma_\textup{p}(T_\om)$ for every $1\neq \alpha\in\BC$ as in that case $k_q=1> 0=k_\la^-+k_\la^0$. For $\la\neq 1$, $\la-\om$ has a zero at $\frac{\al-\al}{\la-1}$ of multiplicity one. For $\lambda = x + iy$  we have $\vert \alpha - \lambda \vert = \vert \lambda - 1 \vert $ if and only if $ (a - x)^2 + (b - y)^2 = (x - 1)^2 + y^2$, which in turn is equivalent to $2by = (a^2 + b^2 - 1) + (2 - 2a)x$. Hence the zero of $\la-\om$ is on $\BT$ precisely when $\la$ is on the line $L$. This shows $\si_\tu{ess}=L$. One easily verifies that the point spectrum and resolvent set correspond to the two half planes indicated above and that these coincide with the images of $\om$ under $\BC\backslash\overline\BD$ and $\BD$, respectively. Since $\la-\om$ can have at most one zero, it is clear from Theorem \ref{T:main2} that $\si_\tu{r}(T_\om)=\emptyset$, so that $\si_\tu{c}(T_\om)=L=\si_\tu{ess}(T_\om)$, as claimed.
\hfill$\Box$
%
\end{example}

\begin{example}\label{E:spectrum4a}
Let $\om(z) = \frac{1}{(z-1)^k}$ for some positive integer $k>1$. Then
\[
\si_\tu{p}(T_\om)=\si(T_\om)=\BC,\quad \si_{r}(T_\om)=\si_\tu{c}(T_\om)=\rho(T_\om)=\emptyset,
\]
and the essential spectrum is given by
\[
\si_\tu{ess}(T_\om)=\om(\BT)=\{(it-\half)^k \mid t\in\BR  \}.
\]


For $k=2$ the situation is as in Figure 2; one can check that the curve $\om(\BT)$ is the parabola $\re(z)=\frac{1}{4}-\im(z)^2$. (Recall that different colors indicate different Fredholm index, as explained at the end of the introduction.)
\begin{figure}
\begin{center}
\includegraphics[height=4cm]{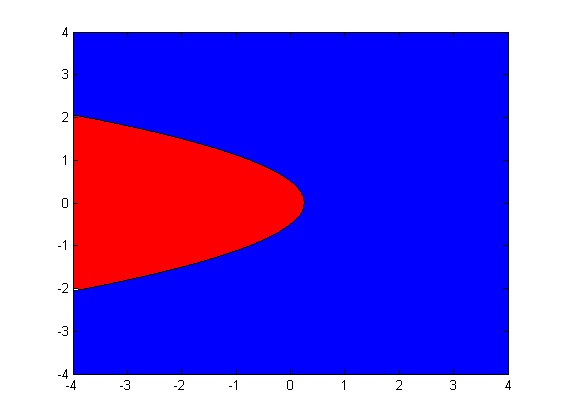}
\caption{Spectrum of $T_\om$ where $\om(z)=\frac{1}{(z-1)^2}$}
\end{center}
\end{figure}

To prove the statements, we start with the observation that for $\vert z\vert = 1$, $\frac{1}{z-1}$ is of the form $it-\frac{1}{2} , t\in\BR$. Thus for $z\in\BT$ with $\frac{1}{z-1}=it-\frac{1}{2}$ we have
\[
\om(z) = \frac{1}{(z-1)^k} = (z-1)^{-k} = (it -\half)^k.
\]
This proves the formula for $\si_\tu{ess}(T_\om)$. For $\la=re^{i\theta}\neq 0$ we have
\[
\la-\om(z)=\frac{\la(z-1)^{k}-1}{(z-1)^k}.
\]
Thus $\la-\om(z)=0$ if and only if $(z-1)^k=\la^{-1}$, i.e., $z=1+r^{-1/k}e^{i(\theta +2\pi l)/k}$ for $l=0,\ldots,k-1$. Thus the zeroes of $\la-\om$ are $k$ equally spaced points on the circle with center 1 and radius $r^{-1/k}$. Clearly, since $k>1$, not all zeroes can be inside $\overline{\BD}$, so $k_q> k_\la^{0}+k_\la^{-}$, and thus $\la\in\si_\tu{p}(T_\om)$. It follows directly from Theorem \ref{T:main2} that $0\in\si_\tu{p}(T_\om)$. Thus $\si_\tu{p}(T_\om)=\BC$, as claimed. The curve $\om(\BT)$ divides the plane into several regions on which the index is a positive constant integer, but the index may change between different regions.
\hfill $\Box$
\end{example}

\section{The essential spectrum need not be connected}\label{S:ExEssSpec}

For a continuous function $\omega$ on the unit circle it is obviously the case that the
curve $\om(\BT)$ is a connected and bounded curve in the complex plane, and hence the
essential spectrum of $T_\omega$ is connected in this case. It was proved by Widom \cite{W64}
that also for $\omega$ piecewise continuous the essential spectrum of $T_\omega$ is connected,
and it is the image of a curve related to $\om(\BT)$ (roughly speaking, filling the jumps with
line segments). Douglas \cite{D98} proved that even for $\omega\in L^\infty$ the essential
spectrum of $T_\omega$ as an operator on $H^2$ is connected.
In \cite{BS06} the question is raised whether or not
the essential spectrum of $T_\omega$ as an operator on $H^p$ is always connected when
$\om \in L^\infty$.

Returning to our case, where $\omega$ is a rational function possibly with poles on the unit circle,
clearly when $\omega$ does have poles on the unit
circle it is not a-priori necessary that $\si_\tu{ess}(T_\om)=\om(\BT)$ is connected. We shall present examples that show that indeed the essential spectrum need not be connected, in contrast with the case where $\omega\in L^\infty$.

Consider $\om=s/q\in\Rat(\BT)$ with $s,q\in\cP$ with real coefficients. In that case $\overline{\om(z)}=\om(\overline{z})$, so that the essential spectrum is symmetric with respect to the real axis. In particular, if $\om(\BT)\cap \BR=\emptyset$, then the essential spectrum is disconnected. The converse direction need not be true, since the essential spectrum can consist of several disconnected parts on the real axis, as the following example shows.

\begin{example}\label{E:disconR}
Consider $\om(z)=\frac{z}{z^2+1}$. Then
\[
\si_\tu{ess}(T_\om)=\om(\BT)=(-\infty,-1] \cup [1,\infty)=\si_\tu{c}(T_\om),\quad
\si_\tu{p}(T_\om)=\BC\backslash \om(\BT),
\]
and thus $\si_\tu{r}(T_\om)=\rho(T_\om)=\emptyset$. Further, for $\la\not\in\om(\BT)$ the Fredholm index is 1.\medskip

Indeed, note that for $z=e^{i\theta}\in\BT$ we have
\[
\om(z)=\frac{1}{z+z^{-1}}=\frac{1}{2\,\re(z)}=\frac{1}{2\cos(\theta)}\in\BR.
\]
Letting $\theta$ run from $0$ to $2\pi$, one finds that $\om(\BT)$ is equal to the union of $(-\infty,-1]$ and $[1,\infty)$, as claimed. Since $\om$ is strictly proper, $\si_\tu{r}(T_\om)=\emptyset$ by Theorem \ref{T:spectrum2}. Applying Theorem \ref{T:recall1} to $T_\om$ we obtain that $T_\om$ is Fredholm with index 1. Hence $T_\om$ is not injective, so that $0\in\si_\tu{p}(T_\om)$. However, since $\BC\backslash \om(\BT)$ is connected, it follows from Theorem \ref{T:main1} that the index of $T_{\la-\om}$ is equal to 1 on $\BC\backslash \om(\BT)$, so that $\BC\backslash \om(\BT)\subset\si_\tu{p}(T_\om)$. However, for $\la$ on $\om(\BT)$ the function $\la-\om$ has two zeroes on $\BT$ as well as two poles on $\BT$. It follows that $\om(\BT)=\si_\tu{c}(T_\om)$, which shows all the above formulas for the spectral parts hold.
\end{example}

As a second example we specify $q$ to be $z^2-1$ and determine a condition on $s$ that guarantees $\si_\tu{ess}(T_\om)=\om(\BT)$ in not connected.

\begin{example}
Consider $\om(z)=\frac{s(z)}{z^2-1}$ with $s\in\cP$ a polynomial with real coefficients. Then for $z\in\BT$ we have
\[
\om(z)=\frac{\overline{z}s(z)}{z-\overline{z}}
=\frac{\overline{z}s(z)}{-2i\,\im(z)}
=\frac{i\overline{z}s(z)}{2\,\im(z)},\quad \mbox{so that}\quad
\im(\om(z))=\frac{\re(\overline{z}s(z))}{2\,\im(z)}.
\]
Hence $\im(\om(z))=0$ if and only if $\re(\overline{z}s(z))=0$. Say $s(z)=\sum_{j=0}^k a_j z^j$. Then for $z\in\BT$ we have
\begin{align*}
\re(\overline{z}s(z)) & = \sum_{j=0}^k a_j \re(z^{j-1}).
\end{align*}
Since $|\re(z^j)|\leq 1$, we obtain that $|\re(\overline{z}s(z))|>0$ for all $z\in\BT$ in case $2|a_1|>\sum_{j=0}^k|a_j|$. Hence in that case $\om(\BT)\cap \BR=\emptyset$ and we find that the essential spectrum is disconnected in $\BC$.

We consider two concrete examples, where this criteria is satisfied.

Firstly, take  $\omega(z)=\frac{z^3+3z+1}{z^2-1}$.  Then
$$
\omega(e^{i\theta})= \frac{1}{2}(2\cos\theta -1) -\frac{i}{2}\frac{2(\cos\theta +1/4)^2+7/4}{\sin\theta},
$$
which is the curve given in Figure 3, that also shows the spectrum and resolvent as well as the essential spectrum.

\begin{figure}
\includegraphics[height=4cm]{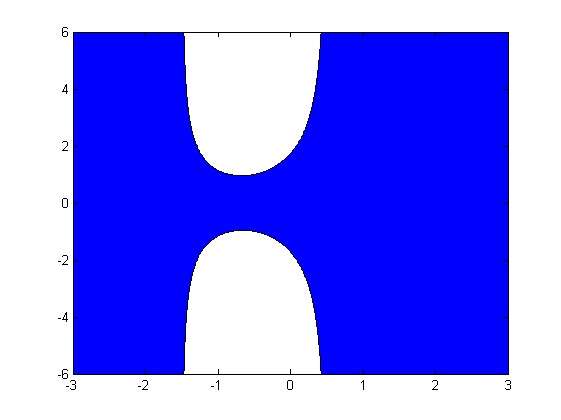}
\caption{Spectrum of $T_\omega$, where $\om(z)=\frac{z^3+3z+1}{z^2-1}$}
\end{figure}

Secondly, take $\omega(z)=\frac{z^4+3z+1}{z^2-1}$. Figure 4 shows the spectrum and resolvent and the essential spectrum. Observe that this is also a case where the resolvent is a bounded set.

\begin{figure}
\includegraphics[height=4cm]{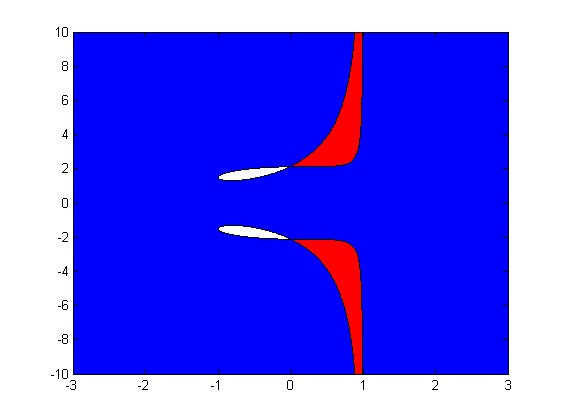}
\caption{Spectrum of $T_\omega$, where $\om(z)=\frac{z^4+3z+1}{z^2-1}$}
\end{figure}

\end{example}

\section{A parametric example}\label{S:Examples2}

In this section we take $\om_k(z) = \frac{z^k + \alpha}{(z - 1)^2}$ for $\alpha\in\BC, \al\neq -1$ and for various integers $k\geq 1$. Note that the case $k=0$ was dealt with in Example \ref{E:spectrum4a} (after scaling with the factor $1+\al$). The zeroes of $\la-\om$ are equal to the roots of
\[
p_{\lambda,\alpha,k}(z)=\lambda q(z)- s(z) = \lambda (z-1)^2 - (z^k + \alpha).
\]
Thus, $\la$ is in the resolvent set $\rho(T_{\om_k})$ whenever $p_{\lambda,\alpha,k}$ has at least two roots in $\BD$ and no roots on $\BT$. Note that Theorem \ref{T:spectrum2} applies in case $k=1,2$. We discuss the first of these two cases in detail, and then conclude with some figures that contain possible configurations of other cases.


\begin{example}\label{E:spectrum4b}
Let $\om(z)=\om_1(z) = \frac{z+\alpha}{(z-1)^2} $ for $\alpha\not = -1$. Then
\begin{equation}\label{EssSpecPara}
\si_\tu{ess}(T_\om)=\om(\BT)=\{(it-\half) + (1+\alpha)(it-\half)^2 \mid t\in\BR\}.
\end{equation}
Define the circle
\[
\BT(-\half,\half)=\{z\in\BC \mid |z+\half|=\half\},
\]
and write $\BD(-\half,\half)$ for the open disc formed by the interior of $\BT(-\half,\half)$ and $\BD^c(-\half,\half)$ for the open exterior of $\BT(-\half,\half)$.

For $\al\notin \BT(-\half,\half)$ the curve $\om(\BT)$ is equal to the parabola in $\BC$ given by
\begin{align*}
\om(\BT) &=\left\{ -(\al+1)(x(y)+i y) \mid y\in\BR \right\},\quad \mbox{ where } \\
x(y) &= \frac{|\al+1|^4}{(|\al|^2+\re(\al))^2}y^2+
\frac{(\re(\al)+1)|\al+1|^2\im(\al)}{(|\al|^2+\re(\al))^2}y+
\frac{|\al|^2(1-|\al|^2)}{(|\al|^2+\re(\al))^2},
\end{align*}
while for $\al\in \BT(-\half,\half)$ the curve $\om(\BT)$ becomes the half line given by
\[
\om(\BT)=\left\{-(\al+1)r - \frac{(\al+1)(1+2\overline{\al})}{4(1-|\al|^2)} \mid r\geq 0 \right\}.
\]
As $\om$ is strictly proper, we have $\si_\tu{r}(T_\om)=\emptyset$. For the remaining parts of the spectrum we consider three cases.
\begin{itemize}
  \item[(i)]
  For $\al\in \BD(-\half,\half)$ the points $-\half$ and $0$ are separated by the parabola $\om(\BT)$ and the connected component of $\BC\backslash \om(\BT)$ that contains $-\half$ is equal to $\rho(T_{\om})$, while the connected component that contains 0 is equal to $\si_\tu{p}(T_\om)$. Finally, $\si_\tu{ess}(T_\om)=\om(\BT)=\si_\tu{c}(T_\om)$.

  \item[(ii)]
  For $\al\in \BT(-\half,\half)$ we have
  \[
  \rho(T_\om)=\emptyset,\quad \si_\tu{c}(T_\om)=\om(\BT)=\si_\tu{ess}(T_\om),\quad \si_\tu{p}(T_\om)=\BC\backslash \om(\BT),
  \]
  and for each $\la\in \om(\BT)$, $\la-\om$ has two zeroes on $\BT$.

  \item[(iii)] For $\al\in \BD^c(-\half,\half)$ we have $\si_\tu{p}(T_\om)=\BC$, and hence $\rho(T_\om)=\si_\tu{c}(T_\om)=\emptyset$.

\end{itemize}

The proof of these statements will be separated into three steps.\smallskip

\paragraph{\it Step 1.}
We first determine the formula of $\om(\BT)$ and show this is a parabola. Note that
\[
\om(z) = \frac{z+\alpha}{(z-1)^2} = \frac{z-1}{(z-1)^2} + \frac{1+\alpha}{(z-1)^2} = \frac{1}{z-1} + (\alpha+1)\frac{1}{(z-1)^2}.
\]
Let $|z|=1$. Then $\frac{1}{z-1}$ is of the form $it-\half$ with $t\in\BR$. So $\om(\BT)$ is the curve
\begin{equation*}
\om(\BT)=\{(it-\half) + (\alpha+1)(it-\half)^2 \mid t\in\BR\}.
\end{equation*}
Thus \eqref{EssSpecPara} holds. Now observe that
\begin{align*}
& (it-\half) + (\alpha+1)(it-\half)^2=\\
&\qquad = -t^2(\alpha+1) + t(i - (\alpha+1)i) + (-\half + \mbox{$\frac{1}{4}$}(\alpha+1))\\
&\qquad  = -t^2(\alpha+1) + (-\alpha i)t + (-\mbox{$\frac{1}{4}$} + \mbox{$\frac{1}{4}$}\alpha)\\
&\qquad = \displaystyle -(\alpha+1)\left(t^2 + t\frac{\alpha i}{\alpha+1} - \frac{1}{4}\left(\frac{\alpha - 1}{\alpha + 1}\right )\right ).
\end{align*}
The prefactor $-(1+\alpha)$ acts as a rotation combined with a real scalar multiplication, so  $\om(\BT)$ is also given by
\begin{equation}\label{omTeq}
\om(\BT)=-(\al+1)\left\{t^2 + t\left(\frac{\alpha i}{\alpha+1}\right ) - \frac{1}{4}\left (\frac{\alpha-1 }{\alpha+1}\right ) \mid t\in\BR\right\}.
\end{equation}
Thus if the above curve is a parabola, so is $\om(\BT)$. Write
\begin{align*}
x(t) &= \re \left(t^2 + t\frac{\alpha i}{1+\alpha} - \frac{1}{4}\left(\frac{\alpha - 1}{\alpha + 1}\right )\right ),\\
y(t) &= \im \left(t^2 + t\frac{\alpha i}{1+\alpha} - \frac{1}{4}\left(\frac{\alpha - 1}{\alpha + 1}\right )\right ).
\end{align*}
Since
\[
\frac{\al i}{\al+1}=\frac{-\im(\al)+i(|\al|^2+\re(\al))}{|\al+1|^2}
\ands
\frac{\al-1}{\al+1}=\frac{(|\al|^2-1)+2i\im(\al)}{|\al+1|^2}
\]
we obtain that
\[
x(t) = t^2-\frac{\im(\al)}{|\al+1|^2}t-\frac{|\al|^2-1}{4|\al+1|^2},\quad
y(t) = \frac{|\al|^2+\re(\al)}{|\al+1|^2}t-\frac{\im(\al)}{2|\al+1|^2}.
\]

Note that $|\al+\half|^2=|\al|^2+\re(\al)+\frac{1}{4}$. Therefore, we have $|\al|^2+\re(\al)=0$ if and only if $|\al+\half|=\half$. Thus $|\al|^2+\re(\al)=0$ holds if and only if $\al$ is on the circle $\BT(-\half,\half)$.

In case $\al\notin \BT(-\half,\half)$, i.e., $|\al|^2+\re(\al)\neq 0$, we can express $t$ in terms of $y$, and feed this into the formula for $x$. One can then compute that
\[
x=\frac{|\al+1|^4}{(|\al|^2+\re(\al))^2}y^2+
\frac{(\re(\al)+1)|\al+1|^2\im(\al)}{(|\al|^2+\re(\al))^2}y+
\frac{|\al|^2(1-|\al|^2)}{(|\al|^2+\re(\al))^2}.
\]
Inserting this formula into \eqref{omTeq}, we obtain the formula for $\om(\BT)$ for the case where $\al\notin \BT(-\half,\half)$.

In case $\al\in \BT(-\half,\half)$, i.e., $|\al|^2+\re(\al)= 0$, we have
\[
|\al+1|^2=1-|\al|^2=1+\re(\al), \quad \im(\al)^2=|\al|^2(1-|\al|^2)
\]
and using these identities one can compute that
\[
y(t)=\frac{-2\im(\al)}{4(1-|\al|^2)}\ands
x(t)=\left(t-\frac{\im(\al)}{2(1-|\al|^2)}\right)^2+\frac{1+2\re(\al)}{4(1-|\al|^2)}.
\]
Thus $\{x(t)+iy(t) \mid t\in\BR\}$ determines a half line in $\BC$, parallel to the real axis and starting in $\frac{1+2\overline{\al}}{4(1-|\al|^2)}$ and moving in positive direction. It follows that $\om(\BT)$ is the half line
\[
\om(\BT)=\left\{-(\al+1)r - \frac{(\al+1)(1+2\overline{\al})}{4(1-|\al|^2)} \mid r\geq 0 \right\},
\]
as claimed.\medskip

\paragraph{\it Step 2.} Next we determine the various parts of the spectrum in $\BC\backslash \om(\BT)$. Since $\om$ is strictly proper, Theorem \ref{T:spectrum2} applies, and we know $\si_\tu{r}(T_\om)=\emptyset$ and $\si_\tu{p}=\om(\BC\backslash \overline{\BD})\cup \{0\}$.

For $k=1$, the polynomial $p_{\la,\al}(z)=p_{\la,\al,1}(z)=\la z^2 -(1+2\la)z+\la-\al$ has roots
\[
\frac{-(1+2\la)\pm \sqrt{1+4\la(1+\al)}}{2\la}.
\]

We consider three cases, depending on whether $\al$ is inside, on or outside the circle $\BT(-\half,\half)$.

Assume $\al\in\BD(-\half,\half)$. Then $\om(\BT)$ is a parabola in $\BC$. For $\la=-\half$ we find that $\la-\om$ has zeroes $\pm i\sqrt{1+2\al}$, which are both inside $\BD$, because of our assumption. Thus $-\half\in\rho(T_\om)$, so that $\rho(T_\om)\neq \emptyset$. Therefore the connected component of $\BC\backslash \om(\BT)$ that contains $-\half$ is contained in $\rho(T_\om)$, which must also contain $\om(\BD)$. Note that $0\in\om(\BT)$ if and only if $|\al|=1$. However, there is no intersection of the disc $\al\in\BD(-\half,\half)$ and the unit circle $\BT$. Thus 0 is in $\si_\tu{p}(T_\om)$, but not on $\om(\BT)$. Hence $0$ is contained in the connected component of $\BC\backslash \om(\BT)$ that does not contain $-\half$. This implies that the connected component containing $0$ is included in $\si_\tu{p}(T_\om)$. This proves our claims for the case $\al\in\BD(-\half,\half)$.

Now assume $\al\in\BT(-\half,\half)$. Then $\om(\BT)$ is a half line, and thus $\BC\backslash \om(\BT)$ consists of one connected component. Note that the intersection of the disc determined by $|\al+\half|<\half$ and the unit circle consists of $-1$ only. But $\al\neq -1$, so it again follows that $0\notin\om(\BT)$. Therefore the $\BC\backslash \om(\BT)=\si_\tu{p}(T_\om)$. Moreover, the reasoning in the previous case shows that $\la=-\half$ is in $\si_\tu{c}(T_{\om})$ since both zeroes of $-\half-\om$ are on $\BT$.

Finally, consider that case where $\al$ is in the exterior of $\BT(-\half,\half)$, i.e.,  $|\al+\half|>\half$. In this case, $|\al|=1$ is possible, so that $0\in\si_\tu{p}(T_\om)$ could be on $\om(\BT)$. We show that $\al=\om(0)\in\om(\BD)$ is in $\si_\tu{p}(T_\om)$. If $\al=0$, this is clearly the case. So assume $\al\neq0$. The zeroes of $\al-\om$ are then equal to $0$ and $\frac{1+2\al}{\al}$. Note that $|\frac{1+2\al}{\al}|> 1$ if and only if $|1+2\al|^2-|\al|^2>0$. Moreover, we have
\[
|1+2\al|^2-|\al|^2=3|\al|^2+4\re(\al)+1=3|\al+\mbox{$\frac{2}{3}$}|^2-\mbox{$\frac{1}{3}$}.
\]
Thus, the second zero of $\al-\om$ is outside $\overline{\BD}$ if and only if $|\al+\frac{2}{3}|^2>\frac{1}{9}$. Since the disc indicated by $|\al+\frac{2}{3}|\leq\frac{1}{3}$ is contained in the interior of $\BT(-\half,\half)$, it follows that for $\al$ satisfying $|\al+\half|>\half$ one zero of $\al-\om$ is outside $\overline{\BD}$, and thus $\om(0)=\al\in \si_\tu{p}(T_\om)$. Note that
\[
\BC=\om(\BC)=\om(\BD)\cup \om(\BT) \cup \om(\BC\backslash \overline{\BD}),
\]
and that $\om(\BD)$ and $\om(\BC\backslash \overline{\BD})$ are connected components, both contained in $\si_\tu{p}(T_\om)$. This shows that $\BC\backslash \om(\BT)$ is contained in $\si_\tu{p}(T_\om)$.\medskip

\paragraph{\it Step 3.} In the final part we prove the claim regarding the essential spectrum $\si_\tu{ess}(T_\om)=\om(\BT)$. Let $\la\in\om(\BT)$ and write $z_1$ and $z_2$ for the zeroes of $\la-\om$. One of the zeroes must be on $\BT$, say $|z_1|=1$. Then $\la\in\si_\tu{p}(\BT)$ if and only if $|z_1z_2|=|z_2|>1$. From the form of $p_{\la,\al}$ determined above we obtain that
\[
\la z^2-(1+2\la)z+\la-\al=\la(z-z_1)(z-z_2).
\]
Determining the constant term on the right hand sides shows that $\la z_1z_2=\la-\al$. Thus
\[
|z_2|=|z_1z_2|=\frac{|\la-\al|}{|\la|}.
\]
This shows that $\la\in\si_\tu{p}(T_\om)$ if and only if $|\la-\al| > |\la|$, i.e., $\lambda$ is in the half plane containing zero determined by the line through $\half\alpha$ perpendicular to the line segment from zero to $\alpha$.

Consider the line given by $|\la-\al| = |\la|$ and the parabola $\om(\BT)$, which is a half line in case $\al\in\BT(-\half,\half)$. We show that $\om(\BT)$ and the line intersect only for $\al\in\BT(-\half,\half)$, and that in the latter case $\om(\BT)$ is contained in the line. Hence for each value of $\al\neq -1$, the essential spectrum consists of either point spectrum or of continuous spectrum, and for $\al\in\BT(-\half,\half)$ both zeroes of $\la-\om$ are on $\BT$, so that $\om(\BT)$ is contained in $\si_\tu{c}(T_\om)$.

As observed in \eqref{EssSpecPara}, the parabola $\om(\BT)$ is given by the parametrization $(it-\half)^2(\alpha+1)+(it-\half)$ with $t\in\BR$, while the line is given by the parametrization
$\half\alpha +si\alpha$ with $s\in\BR$. Fix a $t\in\BR$ and assume the point on $\om(\BT)$ parameterized by $t$ intersects with the line, i.e., assume there exists a $s\in\BR$ such that:
$$
(it-\half)^2(\alpha+1)+(it-\half)=\half\alpha +si\alpha,
$$
Thus
$$
(-t^2-it+\mbox{$\frac{1}{4}$})(\alpha+1)+(it-\half)=\half\alpha +si\alpha,
$$
and rewrite this as
$$
i(-t(\alpha+1)+t-\alpha s)+((-t^2+\mbox{$\frac{1}{4}$})(\alpha+1)-\half -\half \alpha)=0,
$$
which yields
$$
-\alpha i(t+s)+(\alpha+1)(-t^2-\mbox{$\frac{1}{4}$})=0.
$$
Since $t^2+\mbox{$\frac{1}{4}$}>0$, this certainly cannot happen in case $\al=0$. So assume $\al\neq 0$.
Multiply both sides by $-\overline{\alpha}$ to arrive at
$$
|\alpha|^2i (t+s)+(|\alpha|^2+\overline{\alpha})(t^2+\mbox{$\frac{1}{4}$})=0.
$$
Separate the real and imaginary part to arrive at
\[
(|\alpha|^2+\re(\alpha))(t^2+\mbox{$\frac{1}{4}$})+ i(|\al|^2(t+s)-(t^2+\mbox{$\frac{1}{4}$})\im(\al))=0.
\]
Thus
\[
(|\alpha|^2+\re(\alpha))(t^2+\mbox{$\frac{1}{4}$})=0
\ands
|\al|^2(t+s)=(t^2+\mbox{$\frac{1}{4}$})\im(\al).
\]
Since $t^2+\mbox{$\frac{1}{4}$} >0$, the first identity yields $|\alpha|^2+\re(\alpha)=0$, which happens precisely when  $\al\in\BT(-\half,\half)$. Thus there cannot be an intersection when $\al\notin\BT(-\half,\half)$. On the other hand, for $\al\in\BT(-\half,\half)$ the first identity always holds, while there always exists an $s\in\BR$ that satisfies the second equation. Thus, in that case, for any $t\in\BR$, the point on $\om(\BT)$ parameterized by $t$ intersects the line, and thus $\om(\BT)$ must be contained in the line.

We conclude by showing that $\om(\BT)\subset \si_{\tu{p}}(T_\om)$ when $|\al +\half|>\half$ and that $\om(\BT)\subset \si_{\tu{c}}(T_\om)$ when $|\al +\half|<\half$. Recall that the two cases correspond to $|\al|^2+\re(\al)>0$ and $|\al|^2+\re(\al)<0$, respectively. To show that this is the case, we take the point on the parabola parameterized by $t=0$, i.e., take $\la=\frac{1}{4}(\al+1)-\half=\frac{1}{4}(\al-1)$. Then $\la-\al=-\frac{1}{4}(3\al+1)$. So
\[
|\la-\al|^2=\mbox{$\frac{1}{16}$}(9|\al|^2+6\re(\al)+1) \ands
|\la|^2=\mbox{$\frac{1}{16}$}(|\al^2|-2\re(\al)+1).
\]
It follows that $|\la-\al|>|\la|$ if and only if
\[
\mbox{$\frac{1}{16}$}(9|\al|^2+6\re(\al)+1)> |\la|^2=\mbox{$\frac{1}{16}$}(|\al^2|-2\re(\al)+1),
\]
or equivalently,
\[
8(|\al|^2+\re(\al))>0.
\]
This proves out claim for the case $|\la+\half|>\half$. The other claim follows by reversing the directions in the above inequalities.

Figure 5 presents some illustrations of the possible situations.
\begin{figure}
\includegraphics[width=12cm]{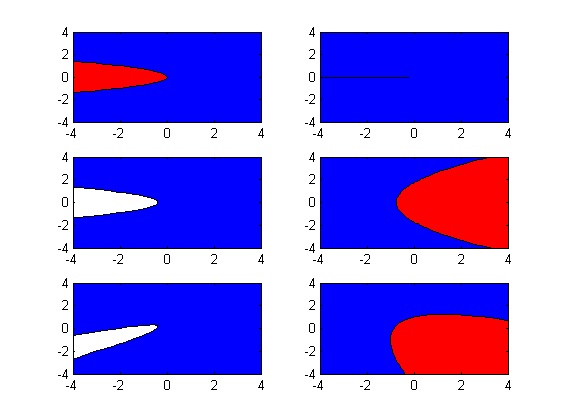}
\caption{Spectrum of $T_\omega$, where $\om(z)=\frac{z+\alpha}{(z-1)^2}$ for some values of
$\alpha$, with $\alpha = 1$, and $\alpha=0$ (top row left and right), $\alpha=1/2$ and $\alpha=-2$ (middle row left and right), $\alpha =-\frac{1}{2}+\frac{1}{4}i$ and $\alpha=-2+i$
(bottom row).}
\end{figure}
 \hfill$\Box$

\end{example}

The case $k=2$ can be dealt with using the same techniques, and very similar results are obtained in that case.

The next examples deal with other cases of $\om_k$, now with $k>2$.

\begin{example}\label{E:spectrum4d}
Let $\om = \frac{z^3 + \alpha}{(z-1)^2}$. Then
{\small
\[
\om(z)   = \frac{z^3 + \alpha}{(z-1)^2} =
(z-1) + 3 + \frac{3}{z-1} + \frac{1+\alpha}{(z-1)^2}.
\]
}
For $z\in\BT$, $\frac{1}{z-1}$ has the form $-\frac{1}{2} + ti, t\in\BR$ and so $\om(\BT)$ has the form
\[
\om(\BT)=\left\{
\frac{1}{-\frac{1}{2} + ti} + 3 + 3(-\frac{1}{2} + ti) + (1+\alpha)\left (- \frac{1}{2} + ti\right)^2,\mid t\in\BR\right\}.
\]
Also $\lambda - \om(z) = \frac{\lambda(z-1)^2 - z^3 - \alpha}{(z-1)^2}$ and so for invertibility we need the polynomial $p_{\lambda,\alpha}(z) = \lambda(z-1)^2 - z^3 - \alpha$ to have exactly two roots in $\BD$.  Since this is a polynomial of degree $3$ the number of roots inside $\BD$ can be zero, one, two or three, and the index of $\lambda-T_\om$ correspondingly can be two, one, zero or minus one. Examples are given in Figure 6.

\bigskip
\begin{figure}
\includegraphics[width=12cm]{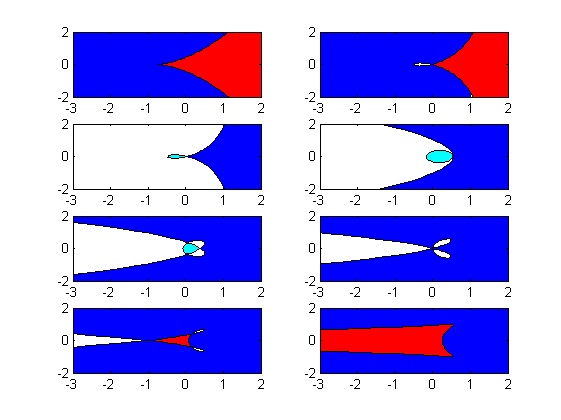}
\caption{Spectrum of $T_\om$ where $\om(z)=\frac{z^3+\alpha}{(z-1)^2}$ for several values of $\alpha$, with $\alpha$ being (left to right and top to bottom) respectively, $-2, -1.05, -0.95, 0.3, 0.7, 1, 1.3, 2$.
}
\end{figure}

\end{example}

\begin{example}\label{E:spectrum5d}
To get some idea of possible other configurations we present some examples with other values of $k$.

For $\om (z)= \frac{z^4 }{(z-1)^2}$ (so $k=4$ and $\alpha =0$) the essential spectrum of $T_\om$ is the curve in Figure 7, the white region is the resolvent set, and color coding for the Fredholm index is as earlier in the paper. For $\om (z)= \frac{z^6 + 1.7}{(z-1)^2}$ (so $k=6$ and $\alpha =1.7$) see Figure 8, and as a final example Figure 9 presents the essential spectrum and spectrum for
$\om(z)=\frac{z^7+1.1}{(z-1)^2}$ and $\om(z)=\frac{z^7+0.8}{(z-1)^2}$. In the latter figure
color coding is as follows: the Fredholm index is $-3$ in the yellow region, $-4$ in the green region and $-5$ in the black region.

\begin{figure}
\includegraphics[height=4cm]{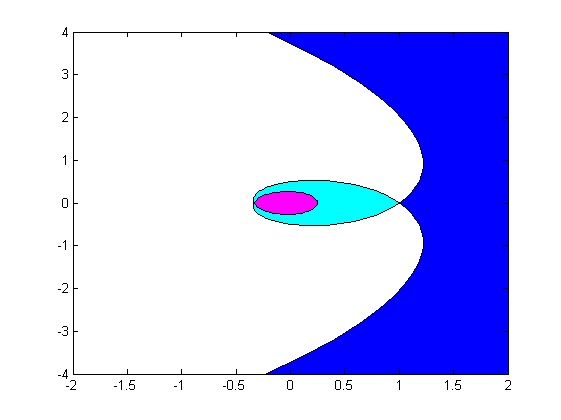}
\caption{The spectrum of $T_\om$, with $k=4$ and $\alpha=0$.}
\end{figure}

\begin{figure}
\includegraphics[height=4cm]{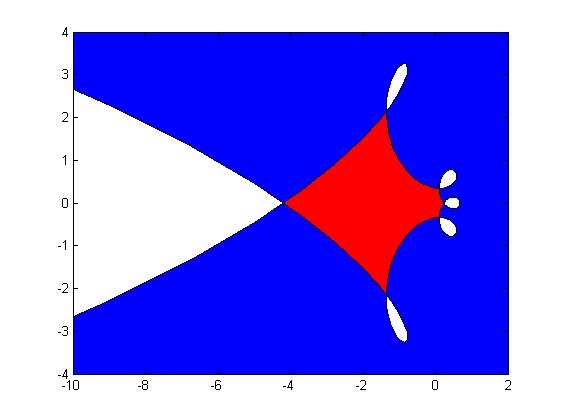}
\caption{The spectrum of $T_\om$ with $k=6$ and $\alpha=1.7$.}
\end{figure}

\begin{figure}
\includegraphics[height=4cm]{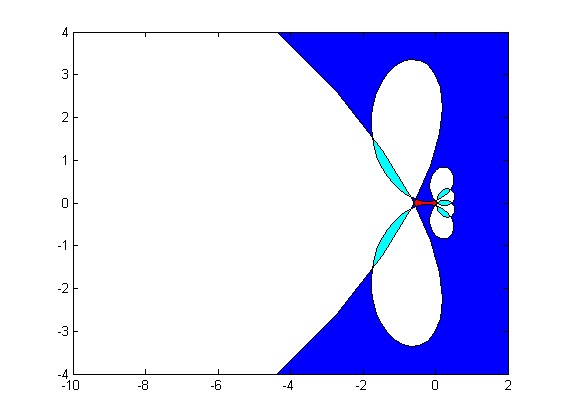}
\includegraphics[height=4cm]{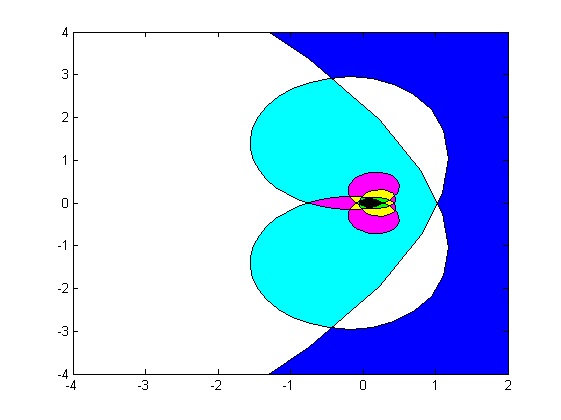}
\caption{The spectrum of $T_\om$ for $k=7$ and $\alpha=1.1$ (left) and $k=7$, $\alpha=0.8$ (right)}
\end{figure}

\end{example}

\paragraph{\bf Acknowledgement}
 The present work is based on research supported in part by the National Research Foundation of South Africa. Any opinion, finding and conclusion or recommendation expressed in this material is that of the authors and the NRF does not accept any liability in this regard.


\begin{thebibliography}{9}

\bibitem{BH64}
A. Brown and P.R. Halmos, Algebraic properties of Toeplitz operators, {\em J. Reine Angw.\ Math.} {\bf 213} (1964), 89--102.

\bibitem{BS06}
A. B\"ottcher and B. Silbermann, {\em Analysis of Toeplitz operators. Second edition}, Springer Monographs in Mathematics, Springer--Verlag, Berlin, 2006.

\bibitem{C67}
D.N. Clark, On the point spectrum on a Toeplitz operator, {\em Trans.\ Amer.\ Math.\ Soc.} {\bf 126} (1967), 251--266.


\bibitem{C90}
J.B. Conway, {\em A course in Functional analysis. Second edition}, Springer, 1990.

\bibitem{D98}
R.G. Douglas, {\em Banach algebra techniques in Operator Theory. Second Edition}, Graduate Texts in Mathematics, 179, Springer, New York, 1988.


\bibitem{FK}
K.H. F\"orster and M.A. Kaashoek.  The asymptotic behaviour of the reduced minumum modulus
of a Fredholm operator. {\em Proc.\ Amer.\ Math.\ Soc.} {\bf 49} (1975), 123--131.

\bibitem{G52}
I. Gohberg, On an application of the theory of the theory normed rings to singular integral equations, {\em Uspehi.\ Matem.\ Nauk} {\bf 7} (1952), 149--156 [Russian].


\bibitem{G67}
I. Gohberg, On Toeplitz matrices composed by Fourier coefficients of piecewise continuous functions, {\em Funkts.\ Anal.\ Prilozh} {\bf 1} (1967), 91--92 [Russian].


\bibitem{GtHJR1}
G.J. Groenewald, S. ter Horst, J. Jaftha and A.C.M. Ran, A Toeplitz-like operator with rational symbol having poles on the unit circle I: Fredholm properties, {\em Oper.\ Theory Adv.\ Appl.}, accepted.


\bibitem{H63}
P.R Halmos, A glimpse into Hilbert space. {\em 1963 Lectures on Modern Mathematics}, Vol.\ I pp.\ 1-–22, Wiley, New York

\bibitem{HW50}
P. Hartman and A. Wintner, On the spectra of Toeplitz's Matrices, {\em Amer.\ J. Math.} {\bf 72} (1950), 359--366.

\bibitem{HW54}
P. Hartman and A. Wintner, The spectra of Toeplitz's matrices,  {\em Amer.\ J. Math.} {\bf 76} (1954), 887--864.

\bibitem{J91}
J. Janas,  Unbounded Toeplitz operators in the Bargmann-Segal space, {\em Studia Math.} {\bf 99} (1991), 87--99.

\bibitem{N}
N.K. Nikolskii, {\em Operators, functions and systems: An easy reading. Vol.\ I:\ Hardy, Hankel and Toeplitz}, American Mathematical Society, Providence, RI, 2002.

\bibitem{S71}
M. Schechter,  {\em Principles of functional analysis}, Academic Press inc., New York, 1971.


\bibitem{W64}
H. Widom, On the spectrum of a Toeplitz operator, {\em Pacific J. Math.} {\bf 14} (1964), 365--375.

\end{thebibliography}
\end{document}